\documentclass[]{amsart}
\usepackage{amssymb}

\usepackage{amsthm,amstext,amscd,latexsym,amsmath,amsfonts,amssymb}
\usepackage{amsmath}
\usepackage{graphicx}

\theoremstyle{plain}
\newtheorem{theorem}{Theorem}
\newtheorem{lemma}[theorem]{Lemma}
\newtheorem{proposition}[theorem]{Proposition}

\theoremstyle{definition}
\newtheorem{definition}[theorem]{Definition}
\newtheorem{corollary}[theorem]{Corollary}
\newtheorem{example}{\sc Example}
\theoremstyle{remark}

\begin{document}
\title[Remarks to G{\l}azek's results on $n$-ary groups]
      {Remarks to G{\l}azek's results on $n$-ary groups}
\author[W. A. Dudek]
       {Wies{\l}aw A. Dudek }
\address{Institute of Mathematics and Computer Science\\
         Wroc{\l}aw University of Technology\\
         Wybrze\.ze Wyspia\'nskiego 27 \\
         50-370 Wroc{\l}aw, Poland}
\email{dudek@im.pwr.wroc.pl}

\begin{abstract}
It is a survey of the results obtained by K. G{\l}azek's and his
co-workers. We restrict our attention to the problems of
axiomatizations of $n$-ary groups, classes of $n$-ary groups,
properties of skew elements and homomorphisms induced by skew
elements, constructions of covering groups, classifications and
representations of $n$-ary groups. Some new results are added too.
 \end{abstract}
 \maketitle

\section{Introduction}

Ternary and $n$-ary generalizations of algebraic structures are
the most natural ways for further development and deeper
understanding of their fundamental properties. First ternary
algebraic operations were introduced already in the XIXth century
by A. Cayley. As a development of Cayley's ideas they were
considered $n$-ary generalization of matrices and their
determinants and general theory of $n$-ary algebras \cite{car5,
poj}, $n$-group rings \cite{zek/art3} and algebras \cite{zek}. For
some physical applications in Nambu mechanics, supersymmetry,
Yang-Baxter equation, etc. see e.g. \cite{vai/ker}. On the other
hand, Hopf algebras and their ternary generalizations play the
basic role in the quantum group theory.

On one of the L. Gluskin's seminars (in '60s of the past century)
B. Gleichgewicht, a friend of K. G{\l}azek, had familiarized
himself with the theory of $n$-ary systems. It was him who brought
the idea of researching such structures to Wroc{\l}aw where at
that time a group of algebraists lead by E. Marczewski was active.
Some time later (in '70s and '80s) a group of mathematicians
interested in $n$-ary systems gathered around the G{\l}azek's and
Gleichgewicht's algebraic seminar at the Institute of Mathematics
of Wroc{\l}aw University. Constructive discussions on this seminar
resulted later in many articles of such authors like (in
alphabetic order): W. A. Dudek, B. Gleichgewicht, J.Michalski, I.
Sierocki, M. B. Wanke-Jakubowska and M.E. Wanke-Jerie. The first
bibliography of $n$-groups and some group-like $n$-ary systems
\cite{Skopje} prepared by K. G{\l}azek in 1983 was based on the
work of this seminar.

Below we present a short survey of the results of K. G{\l}azek's
and his co-workers. We also present few theorems of other authors
and add several new unpublished results. We restrict our attention
to the problems of axiomatizations of $n$-ary groups, classes of
$n$-ary groups, properties of skew elements and homomorphisms
induced by skew elements, constructions of covering groups,
classifications and representations of $n$-ary groups. We finish
our survey with results on independent sets of $n$-ary groups
contained in the article \cite{DG}, which is probably the last
Glazek's article.

\section{Preliminaries}

The non-empty set $G$ together with an $n$-ary operation $f:G^n\to
G$ is called an {\it $n$-ary groupoid} or an {\it $n$-ary
operative} and is denoted by $(G,f)$.

According to the general convention used in the theory of such
groupoids the sequence of elements $x_i,x_{i+1},\ldots,x_j$ is
denoted by $x_i^j$. In the case $j<i$ this symbol is empty. If
$x_{i+1}=x_{i+2}=\ldots=x_{i+t}=x$, then instead of
$x_{i+1}^{i+t}$ we write $\stackrel{(t)}{x}$. In this convention
$f(x_1,\ldots,x_n)= f(x_1^n)$ and
 \[
 f(x_1,\ldots,x_i,\underbrace{x,\ldots,x}_{t},x_{i+t+1},\ldots,x_n)=
 f(x_1^i,\stackrel{(t)}{x},x_{i+t+1}^n) .
 \]
If $\,m=k(n-1)+1$, then the m-ary operation $\,g\,$ of the form
 \[
 g(x_{1}^{k(n-1)+1})=\underbrace{f(f(...f(f}_{k}
(x_{1}^{n}),x_{n+1}^{2n-1}),...),x_{(k-1)(n-1)+2}^{k(n-1)+1})
\]
is denoted by $\,f_{(k)}$ and is called the {\it long product} of
$f$ or an $m$-ary operation {\it derived} from $f$. In certain
situations, when the arity of $\,g\,$ does not play a crucial
role, or when it will differ depending on additional assumptions,
we write $\,f_{(.)}\,$, to mean $\,f_{(k)}\,$ for some
$\,k=1,2,...$.

An $n$-ary groupoid $(G,f)$ is called {\it $(i,j)$-associative} or
an {\it $(i,j)$-associative} if
 \begin{equation}
f(x_1^{i-1},f(x_i^{n+i-1}),x_{n+i}^{2n-1})=
f(x_1^{j-1},f(x_j^{n+j-1}),x_{n+j}^{2n-1})\label{assolaw}
 \end{equation}
holds for all $x_1,\ldots,x_{2n-1}\in G$. If this identity holds
for all $1\leqslant i<j\leqslant n$, then we say that the
operation $f$ is {\it associative} and $(G,f)$ is called an {\it
$n$-ary semigroup} or, in the Gluskin's terminology, an {\it
$n$-ary associative} (cf. \cite{Glu65}). In the binary case (i.e.,
for $n=2$) it is an arbitrary semigroup.

If for all $x_0,x_1,\ldots,x_n\in G$ and fixed
$i\in\{1,\ldots,n\}$ there exists an element $z\in G$ such that
\begin{equation}                                                          \label{solv}
f(x_1^{i-1},z,x_{i+1}^n)=x_0
\end{equation}
then we say that this equation is {\it $i$-solvable} or {\it
solvable at the place $i$}. If this solution is unique, then we
say that (\ref{solv}) is {\it uniquely $i$-solvable}.

An $n$-ary groupoid $(G,f)$ uniquely solvable for all
$i=1,\ldots,n$ is called an {\it $n$-ary quasigroup}. An
associative $n$-ary quasigroup is called an {\it $n$-ary group}.
Note that for $n=2$ it is an arbitrary group.

The idea of investigations of such groups seems to be going back
to E. Kasner's lecture \cite{kasner} at the fifty-third annual
meeting of the American Association for the Advancement of Science
in 1904. But the first paper containing the first important
results was written (under inspiration of Emmy Noether) by W.
D\"ornte in 1928 (cf. \cite{dor}). In this paper D\"ornte observed
that any $n$-ary groupoid $(G,f)$ of the form
 $\, f(x_1^n)=x_1\circ x_2\circ\ldots\circ x_n$, where $(G,\circ
 )$ is a group, is an $n$-ary group but for every $n>2$ there are
$n$-ary groups which are not of this form. $n$-Ary groups of the
first form are called {\it reducible} or {\it derived} from the
group $(G,\circ )$, the second - {\it irreducible}. Moreover in
some $n$-ary groups there exists an element $e$ (called an {\it
$n$-ary identity} or a {\it neutral element}) such that
 \[
 f(\stackrel{(i-1)}{e},x,\stackrel{(n-i)}{e})=x
 \]
holds for all $x\in G$ and for all $i=1,\ldots,n$. It is
interesting that $n$-ary groups containing neutral element are
reducible (cf. \cite{dor}). Irreducible $n$-ary groups do not
contains such elements. On the other hand, there are $n$-ary
groups with two, three and more neutral elements. The set
$Z_n=\{0,1,\ldots,n-1\}$ with the operation
$f(x_1^{n+1})=(x_1+x_2+\ldots +x_{n+1})({\rm mod}\,n)$ is an
example of an $n$-ary group in which all elements are neutral. The
set of all neutral elements of a given $n$-ary group (if it is
non-empty) forms an $n$-ary subgroup (cf. \cite{D'01} or
\cite{gal'03}).

It is worth to note that in the definition of an $n$-ary group,
under the assumption of the associativity of $f$, it suffices to
postulate the existence of a solution of (\ref{solv}) at the
places $i=1$ and $i=n$ or at one place $i$ other than $1$ and $n$.
Then one can prove uniqueness of the solution of (\ref{solv}) for
all $i=1,\ldots,n$ (cf. \cite{post}, p.$213^{17}$).

On the other hand, Sokolov proved in \cite{sok} that in the case
of $n$-ary quasigroups (i.e., in the case of the existence of a
unique solution of (\ref{solv}) for any place $i=1,\ldots,n$) it
is sufficient to postulate the $(j,j+1)$-associativity for some
fixed $j=1,\ldots,n-1$. Basing on the Sokolov's method W. A.
Dudek, K. G{\l}azek and B. Gleichgewicht proved in 1997 the
following proposition (for details see \cite{DGG}).

\begin{proposition}\label{DGG1}{\bf (W.A.Dudek, K.G{\l}azek, B.Gleichgewicht,
1997)}\newline An $n$-ary groupoid $(G,f)$ is an $n$-ary group if
and only if $($at least$)$ one of the following conditions is
satisfied:
\begin{enumerate}
\item [$(a)$] the $(1,2)$-associative law holds and the
equation $(\ref{solv})$ is solvable for $\,i=n\,$ and uniquely
solvable for $\,i=1$,
\item [$(b)$] the $(n-1,n)$-associative law holds and the
equation $(\ref{solv})$ is solvable for $\,i=1\,$ and uniquely
solvable for $\,i=n$,
\item [$(c)$] the $(i,i+1)$-associative law holds for some
$\,i\in \{2,...,n-2\}\,$ and the equation $(\ref{solv})$ is
uniquely solvable for $\,i\,$ and some $j>i$.
\end{enumerate}
\end{proposition}

\medskip

This result was generalized by W. A. Dudek and I. Gro\'zdzi\'nska
(cf. \cite{DD80}) and independently by N. Celakoski (cf.
\cite{cel77}) in the following way:

\begin{proposition} {\bf (N.Celakoski, 1977, W.A.Dudek, I.Gro\'zdzi\'nska, 1979)}\newline
 An $n$-ary semigroup $(G,f)$ is an $n$-ary group if and only if for some
$1\leqslant k\leqslant n-2$ and all $a_1^k\in G$ there are
elements $x_{k+1}^{n-1},\,y_{k+1}^{n-1}\in G$ such that
\[
f(a_1^k,x_{k+1}^{n-1},b)=f(b,y_{k+1}^{n-1},a_1^k)=b\label{DG-80}
\]
for all $\,b\in G$.
\end{proposition}

\medskip
The above two propositions and methods used in the proofs gave the
impulse to further study the axiomatics of $n$-ary groups (cf.
\cite{rus79, rus81, Tyu'85, Usan'03} and many others). From
different results obtained by various authors we select one simple
characterization proved in \cite{gal95}.

\begin{proposition} {\bf (A.M.Gal'mak, 1995)}\newline
An $n$-ary semigroup $(G,f)$ is an $n$-ary group if and only if
for some $1\leqslant i,j\leqslant n-1$ and all $a,b\in G$ there
are $x,y\in G\,$ such that
\[
f(x,\stackrel{(i-1)}{b},\stackrel{(n-i)}{a})=f(\stackrel{(n-j)}{a},\stackrel{(j-1)}{b},y)=b.
\]
\end{proposition}

\bigskip

Note that in some papers there were investigated so-called {\it
infinitary} $(n=\infty)$ semigroups and quasigroups, i.e.,
groupoids $(G,f)$, where for all natural $\,i, j\,$ the operation
$f:G^{\infty}\to G$ satisfies the identity
 \[
f(x_1^{i-1},f(x_i^{\infty}),y_1^{\infty})=
f(x_1^{j-1},f(x_j^{\infty}),y_1^{\infty})
 \]
and the equation $\,f(x_1^{k-1},z_k,x_{k+1}^{\infty})=x_0\,$ has a
unique solution $z_k$ at any place $k$.

From the general results obtained in \cite{belz} and \cite{MTC}
one can deduce that infinitary groups have only one element. Below
we present a simple proof of this fact.

If $(G,f)$ is an infinitary group, then, according to the
definition, for any $y,z\in G$ and $u=f(\stackrel{(\infty)}{y})$
there exists $x\in G$ such that
$z=f(u,y,x,\stackrel{(\infty)}{y})$. Thus
\[\arraycolsep.5mm
 \begin{array}{rl}
f(z,\stackrel{(\infty)}{y})=&f(f(u,y,x,\stackrel{(\infty)}{y}),\stackrel{(\infty)}{y})=
 f(u,y,f(x,\stackrel{(\infty)}{y}),\stackrel{(\infty)}{y})\\[6pt]
=&f(f(\stackrel{(\infty)}{y}),y,f(x,\stackrel{(\infty)}{y}),\stackrel{(\infty)}{y})
=f(y,f(\stackrel{(\infty)}{y}),y,f(x,\stackrel{(\infty)}{y}),\stackrel{(\infty)}{y})\\[6pt]
=&f(y,u,y,f(x,\stackrel{(\infty)}{y}),\stackrel{(\infty)}{y})
=f(y,f(u,y,x,\stackrel{(\infty)}{y}),\stackrel{(\infty)}{y})=f(y,z,\stackrel{(\infty)}{y}),
 \end{array}
 \]
i.e., for all $y,z\in G\,$ we have
\[
f(z,\stackrel{(\infty)}{y})=f(y,z,\stackrel{(\infty)}{y}).
\]
Using this identity and the fact that for all $\,x,y\in G$ there
exists $z\in G$ such that $\,x=f(z,\stackrel{(\infty)}{y}),\,$ we
obtain
\[\arraycolsep.5mm
 \begin{array}{rl}
f(\stackrel{(\infty)}{x})=&f(x,f(z,\stackrel{(\infty)}{y}),\stackrel{(\infty)}{x})=
 f(x,f(y,z,\stackrel{(\infty)}{y}),\stackrel{(\infty)}{x})\\[6pt]
=&f(x,y,f(z,\stackrel{(\infty)}{y}),\stackrel{(\infty)}{x})
=f(x,y,\stackrel{(\infty)}{x}),
 \end{array}
 \]
which, by the uniqueness of the solution at the second place,
implies $x=y$. So, $G$ has only one element.

\medskip

To avoid repetitions in the sequel we consider only the case when
$n$ is a natural number higher than $2$, but a part of our results
is also true for $n=2$.

\section{Varieties of $n$-ary groups}

Directly from the definition of an $n$-ary group $(G,f)$ we can
see that for every $x\in G$ there exists only one $z\in G$
satisfying the equation
\begin{equation}
f(\overset{(n-1)}{x},z)=x.  \label{skew}
\end{equation}
This element is called \textit{skew} to $x$ and is denoted by
$\overline{x}$. In a ternary group ($n=3$) derived from the binary
group $(G,\circ )$ the skew element coincides with the inverse
element in $(G,\circ )$. Thus, in some sense, the skew element is
a generalization of the inverse element in binary groups. This
suggests that for $n\geqslant 3$ any $n$-ary group $(G,f) $ can be
considered as an algebra $(G,f,\bar{\,}\;)$ with two operations:
one $n$-ary $\,f:G^{n}\rightarrow G$ and one unary
$\;\bar{\,}:x\rightarrow \overline{x}$.

In ternary groups, as it was proved by W. D\"{o}rnte (cf.
\cite{dor}), we have
$\overline{f(x,y,z)}=f(\overline{z},\overline{y},\overline{x})$
and $\overline{\overline{x}}=x$, but for $n>3$ it is not true. For
$n>3$ there are $n$-ary groups in which one fixed element is skew
to all elements (cf. \cite{D'90}) and $n$-ary groups in which any
element is skew to itself. Then, of course,
$f(\overset{(n)}{x})=x$ for every $x\in G$. Such $n$-ary groups
are called \textit{idempotent}.

Nevertheless, the concept of skew elements plays a crucial role in
the theory of $n$-ary groups. Namely, as W. D\"ornte proved, in
any $n$-ary group $(G,f)$ for all $\,x,y\in G$, $\,2\leqslant
i,j\leqslant n\,$ and $\,1\leqslant k\leqslant n$ we have
 \begin{eqnarray}
  f(\stackrel{(i-2)}{x},\overline{x},\stackrel{(n-i)}{x},y)=y, \label{dor-r}\\
  f(y,\stackrel{(n-j)}{x},\overline{x},\stackrel{(j-2)}{x})=y. \label{dor-l}
  \end{eqnarray}

These two identities, called now {\it D\"ornte's identities}, are
used by many authors in description of a class of $n$-ary groups.

\begin{theorem} {\bf (B.Gleichgewicht, K.G{\l}azek, 1967)}\label{GG67}\newline
An algebra $(G,f,\bar{{\,}}\;)$ with one associative $n$-ary
$(n>2)$ operation $f$ and one unary operation
$\bar{{\,}}:x\mapsto\overline{x}$ is an $n$-ary group if and only
if the conditions $(\ref{dor-r})$ and $(\ref{dor-l})$ are
satisfied for all $x,y\in G$ and $i=j=2,3$.
\end{theorem}

It is the first important characterization of the variety of
$n$-ary groups. For example, basing on this theorem it is not
difficult to see that the function
$$
h(x,y,z)=f(x,\overline{y},\stackrel{(n-3)}{x},z)
$$
is the so-called Mal'cev operation. So, the class of all
$n$-groups (for any fixed $n>2$) is a Mal'cev variety, all
congruences of a given $n$-ary group commutes and the lattice of
all congruences of a fixed $n$-ary group is modular. Moreover, the
generalized Zassenhaus Lemma and the generalized Schreier and
H\"older-Jordan Theorems hold in any $n$-ary group (cf.
\cite{GG'77}). Schreier varieties of $n$-ary groups are described
by V. A. Artamonov (see \cite{Art'79}).

Unfortunately, the above system of identities defining the variety
of $n$-ary groups is not independent. The first independent system
of identities selecting the variety of $n$-ary groups from the
variety of $n$-ary semigroups was given by K. G{\l}azek and his
coauthors ten years later (cf. \cite{DGG}). Below we present the
minimal system of identities defining this variety. It is the main
result of \cite{D'80}.

\begin{theorem} {\bf (W.A.Dudek, 1980)}\label{rem}\newline
The class of $n$-ary groups coincides with the variety of $n$-ary
groupoids $(G,f,\bar{}\;)$ with an unary operation
$\,\bar{}:x\to\overline{x}$ for some fixed $i,j\in\{2,\ldots,n\}$
satisfying the identities $(\ref{dor-r})$, $(\ref{dor-l})$ and the
identity
\[
 f(f(x_1^{n}),x_{n+1}^{2n-1})=f(x_1,f(x_2^{n+1}),x_{n+2}^{2n-1}).
 \]
\end{theorem}

\medskip

Theorem~\ref{rem} is valid for $n>2$, but this theorem can be
extended to the case $n=2$. Namely, let
$\widehat{}:x\to\widehat{x}$ be an unary operation, where
$\widehat{x}$ is a solution of the equation
$f_{(2)}(\stackrel{(2n-2)}{x},\widehat{x})=x$. Then using the same
method as in the proof of Theorem 2 in \cite{DGG} we can prove the
following result announced in \cite{D'80}.

\begin{theorem} {\bf (W.A.Dudek, 1980)}\newline
The class of $n$-ary $(n\geqslant 2)$ groups coincides with the
variety of algebras $(G,f,\widehat{}\;)$ with one associative
$n$-ary operation $f$ and one unary operation
$\,\widehat{}:x\to\widehat{x}$ satisfying for some fixed
$i,j\in\{2,\ldots,n\}$ the following identities
\[
 f_{(2)}(y,\stackrel{(i-2)}{x},\widehat{x},\stackrel{(2n-i-1)}{x})=y=
f_{(2)}(\stackrel{(2n-1-j)}{x},\widehat{x},\stackrel{(j-2)}{x},y).
 \]
\end{theorem}

\medskip

Other systems of identities defining the variety of $n$-ary groups
one can find in \cite{D'95} and \cite{Usan'03}.

\section{Skew elements}

An {\it $n$-ary power} of $x$ in an $n$-ary group $(G,f)$ is
defined in the following way: $x^{<0>}=x$
 and
$x^{<k+1>}=f(\overset{(n-1)}{x},x^{<k>})$ for all $k>0$. \
$x^{<-k>}$ is an element $z$ such that
$f(x^{<k-1>},\overset{(n-2)}{x},z)=x^{<0>}=x$. Then
$\overline{x}=x^{<-1>}$, $\overline{\overline{x}}=x^{<n-3>}$ and
\begin{equation}
f(x^{<k_1>},\ldots,x^{<k_n>})=x^{<k_1+\ldots+k_n+1>} \label{p1}
\end{equation}
\begin{equation}
(x^{<k>})^{<t>}=x^{<kt(n-1)+k+t>}.\label{p2}
\end{equation}

Now, putting $\overline{x}^{(0)}=x$ and denoting by
$\overline{x}^{(s+1)}$ the skew element to $\overline{x}^{(s)}$,
we obtain the sequence of elements: $x,\, \overline{x}^{(1)},
\overline{x}^{(2)},\overline{x}^{(3)},\overline{x}^{(4)}$ and so
on. In a $4$-ary group derived from the additive group
$\mathbb{Z}_8$ we have $\overline{x}\equiv 6x({\rm mod }\,8)$,
$\overline{\overline{x}}\equiv 4x({\rm mod}\, 8)$ and
$\overline{x}^{(s)}=0$ for every $s>2$, but in an $n$-ary group
derived from the additive group of integers
$\overline{x}^{(s)}\ne\overline{x}^{(t)}$ for all $s\ne t$. Any
subgroup containing $x$ contains also $\overline{x}$ and all
$\overline{x}^{(s)}$. The order of the smallest subgroup
containing $x$ is called the {\it $n$-ary order of $x$} and is
denoted by ord$_n(x)$. It is the smallest positive integer $k$
such that $x^{<k>}=x$ (cf. \cite{post}). Obviously

\begin{center}
ord$_n(x)\geqslant{\rm ord}_n({\overline{x}})\geqslant{\rm
ord}_n(\overline{x}^{(2)})\geqslant{\rm
ord}_n(\overline{x}^{(3)})\geqslant\ldots$
\end{center}
In fact ord$_n(\overline{x})$ is a divisor of ord$_n(x)$.

In connection with this K. G{\l}azek posed in 1978 the following
question:

\medskip\noindent
{\bf Question 1.} {\it When ord$_n(x)={\rm ord}_n({\overline{x}})$
?}

\medskip The first partial answer was given in \cite{ww'79}: {\it If an $n$-ary
group $(G,f)$ has a finite order relatively prime to $n-2$, then
ord$_n(x)={\rm ord}_n({\overline{x}})$ \ for all $x\in G$}. The
full answer was found two years later (cf. \cite{DD'81}): {\it in
the case when ord$_n(x)$ is finite, ord$_n(x)={\rm
ord}_n({\overline{x}})$ if and only if ord$_n(x)$ and $n-2$ are
relatively prime}.

Examples of infinite $n$-ary groups in which all elements have the
same $n$-ary order $k>1$ are given in \cite{auto}. Such groups are
a set-theoretic union of disjoint isomorphic subgroups of order
$k$. The skew element is idempotent, i.e.,
ord$_n(\overline{x})=1$, if and only if ord$_n(x)$ is a divisor of
$n-2$ (cf. \cite{DD'81}). If an $n$-ary group has a finite order
$g$ and every prime divisor of $g$ is a divisor of $n-2$, then for
every element $x$ of this group there exists a natural number $t$
such that ord$_n(\overline{x}^{(t)})=1$.

\begin{theorem} {\bf (I.M.Dudek, W.A.Dudek, 1981)}\newline
Let ord$_n(x)=p_1^{\beta_1}p_2^{\beta_2}\ldots p_m^{\beta_m}$,
where $p_1,p_2\ldots,p_m$ are primes. Then
$\lim\limits_{t\to\infty}{\rm ord}_n(\overline{x}^{(t)})=1$ or
$\lim\limits_{t\to\infty}{\rm
ord}_n(\overline{x}^{(t)})=p_1^{\beta_1}p_2^{\alpha_2}\ldots
p_k^{\alpha_k}$, where primes $p_1,p_2\ldots,p_k$, $k\leqslant m$,
are not divisors of $n-2$.
\end{theorem}

In \cite{auto} it is proved that $\overline{x}^{(m)}=x^{<S_m>}$,
where $S_m=\frac{(2-n)^m-1}{n-1}$. So, $\overline{x}^{(m)}=x$ if
and only if ord$_n(x)$ divides $S_m$. The natural question is:
{\it for which $n$-ary groups there exists fixed $m\in\mathbb{N}$
such that $\overline{x}^{(m)}=\overline{y}^{(m)}$ for all $x,y\in
G$} (cf. \cite{D'88}). For $m=1$ the full answer is given in
\cite{D'90}. For $m>1$ we have only a partial answer. Namely, an
$n$-ary group $(G,f)$ in which
$\overline{x}^{(m)}=\overline{y}^{(m)}$ holds for all $x,y\in G$
is torsion free and its exponent is a divisor of $S_m^2-S_m$ (cf.
\cite{sokh}).

An $n$-ary group is said to be {\it semiabelian} if it satisfies
the identity
\[
f(x_1^n)=f(x_n,x_2^{n-1},x_1) .
\]
The class of all semiabelian $n$-ary groups coincides  with the
class of {\it medial} ({\it entropic}) $n$-ary groups, i.e.,
$n$-ary groups satisfying the identity
\[
f(f(x_{11}^{1n}),f(x_{21}^{2n}),\ldots,f(x_{n1}^{nn}))=
f(f(x_{11}^{n1}),f(x_{12}^{n2}),\ldots,f(x_{1n}^{nn})).
\]
This means that the matrix $[x_{ij}]_{n\times n}$ can be read by
rows or by columns.

Some authors use also the term {\it abelian} instead of {\it
semiabelian} and consider such $n$-ary groups as a special case of
the so-called abelian (commutative) general algebras. This implies
that every $n$-ary subgroup of a semiabelian $n$-ary group is a
block of some congruence of this group and the lattice of all
$n$-ary subgroups of this group is modular.

An $n$-ary group $(G,f)$ is semiabelian if and only if there
exists an element $a\in G$ such that
$f(x,\stackrel{(n-2)}{a},y)=f(y,\stackrel{(n-2)}{a},x)$ for all
$x,y\in G$ (cf. \cite{D'80}). This means that for $n=2$, a
semiabelian $n$-ary (i.e., binary) group is commutative. For $n>2$
a commutative $n$-ary group is defined as a group in which
$f(x_1^n)=f(x_{\sigma(1)},x_{\sigma(2)},\ldots,x_{\sigma(n)})$ for
any permutation $\sigma$ of $\{1,2,\ldots,n\}$. The class of
commutative $n$-ary groups is described by J. Timm (cf.
\cite{Timm}).

K. G{\l}azek and B. Gleichgewicht observed in \cite{GG'77} that in
any semiabelian $n$-ary group we have
\begin{equation}
\overline{f(x_1^n)}=f(\overline{x}_1,\overline{x}_2,\ldots,\overline{x}_n),
\label{shom}
\end{equation}
which means that in these $n$-groups the operation
$\,\bar{}:x\to\overline{x}$ is an endomorphism. Also
$h^s(x)=\overline{x}^{(s)}$ is an endomorphism. The converse is
not true. So, in semiabelian $n$-ary groups
\[
 G^{(s)}=\{\overline{x}^{(s)}|\,x\in G \}
\]
is an $n$-ary subgroup for every natural $s$. (It is an $n$-ary
subgroup in any $n$-ary group satisfying (\ref{shom}), not
necessary semiabelian.) Obviously $G\supset G^{(1)}\supset
G^{(2)}\ldots$ It is clear that for any finite $n$-ary group
$(G,f)$ there exists $t\in\mathbb{N}$ such that $G^{(s)}=G^{(t)}$
for $s\geqslant t$. On the other hand, an $n$-ary group derived
from the additive group of integers is a simple example of $n$-ary
group in which $G^{(s)}\ne G^{(t)}$ for all $s\ne t$ ($G^{(s)}$
contains all integers divided by $(n-2)^s$. The question (Problem
5 in \cite{D'88}) on the characterization of classes of $n$-ary
groups satisfying the descending chain condition for $G^{(s)}$ is
open.

Note that in some $n$-ary groups $(G,f)$ the operation
$\,\bar{}:x\to\overline{x}$ induces a cyclic subgroup in the group
of all automorphisms of $(G,f)$ (cf. \cite{D'88}). Moreover this
subgroup is invariant in the group $Aut(G,f)$ and in the group of
all splitting-automorphisms in the sense of P{\l}onka (cf.
\cite{Plo}), i.e., automorphisms satisfying the identity
$h(f(x_1^n))=f(x_1^{i-1},h(x_i),x_{i+1}^n)$.

\medskip
The natural question is:

\medskip\noindent
{\bf Question 2.} {\it When the operation
$\,\bar{}:x\to\overline{x}$ is an endomorphism?}

\medskip The first answer was given in \cite{D'01}:

\begin{proposition} {\bf (W.A.Dudek, 2001)}\label{P8}\newline
The operation $\,\bar{}:x\to\overline{x}$ is an endomorphism of an
$n$-ary group $(G,f)$ if and only if there exists an element $a\in
G$ such that

\begin{enumerate}
\item[$(i)$] \ $\overline{f(x,a,\ldots,a,y)}=
f(\overline{x},\overline{a},\ldots,\overline{a},\overline{y})$,
\item[$(ii)$] \ $\overline{f(\overline{a},x,a,\ldots,a)}=
f(\overline{\overline{a}},\overline{x},\overline{a},\ldots,\overline{a})$,
\item[$(iii)$] \ $\overline{f(\overline{a},\overline{a},\ldots,\overline{a})}=
f(\overline{\overline{a}},\overline{\overline{a}},\ldots,\overline{\overline{a}})$
\end{enumerate}
for all $x,y\in G$.
\end{proposition}

The last condition can be omitted. Indeed, using (\ref{p1}),
(\ref{p2}) and the fact that $\overline{a}=a^{<-1>}$ and
$\overline{\overline{a}}=a^{<n-3>}$ for every $a\in G$, it is not
difficult to see that the left and right side of $(iii)$ are equal
to $a^{<n^2-3n+1>}.$ So, $(iii)$ is valid in any $n$-ary group.

\begin{corollary}{\it
The operation $\,\bar{}:x\to\overline{x}$ is an endomorphism of an
$n$-ary group $(G,f)$ if and only if the equations

\begin{enumerate}
\item[$(i)$] \ $\overline{f(x,a,\ldots,a,y)}=
f(\overline{x},\overline{a},\ldots,\overline{a},\overline{y})$,
\item[$(ii)$] \ $\overline{f(\overline{a},x,a,\ldots,a)}=
f(\overline{\overline{a}},\overline{x},\overline{a},\ldots,\overline{a})$,
\end{enumerate}
hold for all $x,y\in G$ and some fixed $a\in G$.}
\end{corollary}

Another answer is given in \cite{sokh} and \cite{Shch}:

\begin{proposition}\label{P9} {\bf (F.M.Sokhatsky,
2003)}\newline The operation $\,\bar{}:x\to\overline{x}$ is an
endomorphism of an $n$-ary group $(G,f)$ if and only if the
following two identities are satisfied:
\[\arraycolsep=.5mm
\begin{array}{rl}
f(\stackrel{(n-1)}{u},f(\stackrel{(n-2)}{x},u,u))=&
f(f(\stackrel{(n-2)}{x},u,u),\stackrel{(n-1)}{u}),\\[2pt]
f(f(x,\!\!\stackrel{(n-2)}{u}\!\!,y),\dots,f(x,\!\!\stackrel{(n-2)}{u}\!\!,y),u,u)=&
f(\stackrel{(n-2)}{y}\!\!,f(u,f(x,\!\!\stackrel{(n-1)}{u}),\dots,
f(x,\!\!\stackrel{(n-1)}{u}),x,u),u).
\end{array}
\]
\end{proposition}

\begin{theorem} {\bf (N.A.Shchuchkin,
2006)}\newline For odd $k$, the operation
$\,\overline{\;\,\rule{0pt}{4pt}}^{(k)}:x\to\overline{x}^{(k)}$ is
an endomorphism of an $n$-ary group $(G,f)$ if and only if the
identity:
\[\begin{array}{l}
f_{(n-1)}(x_1,\overset{((n-2)^k)}{x_2},\overset{((n-2)^k)}{x_3},\ldots,\overset{((n-2)^k)}{x_{n+1}},
x_{n+2})=\\[5pt]
\rule{22mm}{0mm}f(x_1,\underbrace{f(x_{n+1},x_n,\!...,x_2),\ldots,
f(x_{n+1},x_n,\!...,x_2)}_{(n-2)^k\ \ times},x_{n+2})\end{array}
\]
is satisfied.
\end{theorem}

\begin{theorem} {\bf (N.A.Shchuchkin,
2006)}\newline For even $k$, the operation
$\,\overline{\;\,\rule{0pt}{4pt}}^{(k)}:x\to\overline{x}^{(k)}$ is
an endomorphism of an $n$-ary group $(G,f)$ if and only if the
identity:
\[
f_{(\cdot)}(\underbrace{f(x_1^n),\ldots,f(x_1^n)}_{(n-2)^k})=
f_{(\cdot)}(\overset{((n-2)^k)}{x_1},\overset{((n-2)^k)}{x_2},
\ldots,\overset{((n-2)^k)}{x_n})
\]
is satisfied.
\end{theorem}
\section{Hossz\'u-Gluskin algebras}

Let $(G,f)$ be an $n$-ary group. Fixing in $f(x_{1}^{n})$ some
$k<n$ elements we obtain a new $(n-k)$-ary operation which in
general is not associative. It is associative only in the case
when these fixed elements are located in some special places (cf.
\cite{DM'82}).

Binary operations of the form $x\ast y=f(x,a_2^{n-1},y)$, where
elements $a_2,\ldots,a_{n-1}\in G$ are fixed, play a very
important role. It is not difficult to see that $(G,\ast)$ is a
group. Fixing different elements $a_2,\ldots,a_{n-1}$ we obtain
different groups. Since all these groups are isomorphic (cf.
\cite{DM'82}) we can consider only one group $(G,\circ)$, where
$x\circ y=f(x,\overset{(n-2)}{a},y)$. This group is denoted by
$ret_a(G,f)$ and is called a \textit{binary retract} of $(G,f)$.
The identity of this group is $\overline{a}$. The inverse element
to $x$ has the form
\[
x^{-1}=f(\overline{a},\overset{(n-3)}{x},\overline{x},\overline{a}).
\]
An $n$-ary group $(G,f)$ is semiabelian only in the case when its
binary retract $ret_a(G,f)$ is commutative.

The strong connection between $n$-ary groups and their binary
retracts was observed for the first time in 1963 by M. Hossz\'u
(cf. \cite{Hosszu}). He proved the following theorem:

\begin{theorem}\label{thGH}
 An $n$-ary groupoid $(G,f)$, $n>2$, is an $n$-ary group if and only if

\begin{enumerate}
\item[$(i)$]  on $G$ one can define a binary operation $\cdot$ such that $(G,\cdot)$
is a group,
\item[$(ii)$]  there exist an automorphism $\varphi$ of $(G,\cdot)$ and $b\in G$
such that $\varphi(b)=b$,
\item[$(iii)$]  $\varphi^{n-1}(x)=b\cdot x\cdot b^{-1}$ holds for every $x\in G$,
\item[$(iv)$]  $f(x_{1}^{n})=x_{1}\cdot \varphi (x_{2})\cdot \varphi
^{2}(x_{3})\cdot \varphi ^{3}(x_{4})\cdot \ldots \cdot \varphi
(x_{n})^{n-1}\cdot b$ for all $x_{1},\ldots ,x_{n}\in G$.
\end{enumerate}
\end{theorem}

\medskip

Two years later, this theorem was proved by L. M. Gluskin in more
general form (cf. \cite{Glu65}). Now this theorem is known as the
Hossz\'u-Gluskin Theorem. Some important generalization of this
theorem one can find in \cite{DM'82}, \cite{sok} and
\cite{Usan'03}.

The algebra $(G,\cdot,\varphi,b)$ of the type $(2,1,0)$, where
$(G,\cdot)$ is a (binary) group, $b\in G$ is fixed, $\varphi\in
Aut(G,\cdot)$, $\varphi(b)=b$ and $\varphi^{n-1}(x)=b\cdot x\cdot
b^{-1}$ for every $x\in G$ is called a \textit{Hossz\'u-Gluskin
algebra} (briefly: an \textit{$HG$-algebra}). We say that an
$HG$-algebra $(G,\cdot,\varphi,b)$ is \textit{associated} with an
$n$-ary group $(G,f)$ if the the last condition of Theorem
\ref{thGH} is satisfied. In this case we also say that an $n$-ary
group $(G,f)$ is \textit{$\langle\varphi,b\rangle$-derived} from
the group $(G,\cdot)$. If $\varphi(x)=x$ for every $x\in G$ we say
that $(G,f)$ is {\it $b$-derived} from $(G,\cdot)$. The
systematical study of connections between $n$-ary groups
$b$-derived from given binary group was initiated by K. G{\l}azek
and J. Michalski in \cite{GM'83} and continued by J. Michalski, W.
A. Dudek, M. Pop and many others. Connections between $(G,f)$ and
$n$-ary groups $\langle\varphi,b\rangle$-derived from $ret_a(G,f)$
are described by W. A. Dudek and J. Michalski in \cite{DM'82,
DM'84, DM'85}. All commutative $n$-ary groups are $b$-derived from
some of their retracts (cf. \cite{Timm}).

Let $\mathfrak{G}=(G,f,\bar{}\;)$ be a semiabelian $n$-ary group.
Then the $HG$-algebra associated with $\mathfrak{G}$ has a
commutative group operation denoted by $+$.

Let $\mathfrak{H}=(G,+,\varphi,b)$ be associated with
$\mathfrak{G}$ and $\mathfrak{G}_a=(G,f,\bar{}\;,a)$. Then
$\mathfrak{H}$ and $\mathfrak{G}_a$ are term equivalent (cf.
\cite{DG}) and
\[
\arraycolsep=.5mm
\begin{array}{rl}
-y&=f(\overline{a},\stackrel{(n-3)}{y},\overline{y},\overline{a}\,),\\[4pt]
x+y&=f(x,\stackrel{(n-3)}{(-y)},\overline{(-y)},\overline{a}\,),\\[4pt]
\varphi(x)&=f(\overline{a},x,\stackrel{(n-2)}{a}),\\[4pt]
{\rm and } \ \ \ \ b&=f(\stackrel{(n)}{\overline{a}}).
\end{array}
\]

To describe all term operations of $\mathfrak{G}_a$ by using the
language of $HG$-algebras we denote by $g_i$ the following
operation
\begin{equation}\label{g_i}
g_i(x)=k_{i1}\varphi^{l_{i1}}(x)+k_{i2}\varphi^{l_{i2}}(x)+\ldots
+k_{it}\varphi^{l_{it}}(x),
\end{equation}
where $t,l_{i1},\ldots,l_{it}$ are  non-negative integers and
$k_{i1},\ldots,k_{it}\in\mathbb{Z}$.

\begin{theorem} {\bf (W.A.Dudek, K.G{\l}azek, 2006)}\newline
Let $\,\mathfrak{H}=(G,+,\varphi,b)$ be the $HG$-algebra
associated with a semiabelian $n$-ary group $\mathfrak{G}$. Then
all unary term operations of $\mathfrak{H}$ $($and of
$\;\mathfrak{G}_a\,)$ are of the form
\[
g(x)=g_i(x)+k_g b
\]
for some $g_i$ of the form $(\ref{g_i})$ and $k_g\in\mathbb{Z}$.
\end{theorem}

\begin{theorem}{\bf (W.A.Dudek, K.G{\l}azek, 2006)}\newline
Let $\,\mathfrak{H}=(G,+,\varphi,b)$ be the $HG$-algebra
associated with a semiabelian $n$-ary group $\mathfrak{G}$. Then
all $m$-ary term operations of $\mathfrak{H}$ $($and of
$\;\mathfrak{G}_a\,)$ are of the form
\begin{equation}
F(x_1,\ldots,x_m)=\sum\limits_{i=1}^{m}g_i(x_i)+k_F b\label{sum}
\end{equation}
for some $g_i$ of the form $(\ref{g_i})$ and $k_F\in\mathbb{Z}$.
\end{theorem}

The Hossz\'u-Gluskin Theorem was used by K. G{\l}azek and his
co-workers to calculation of $n$-ary groups of some special types.
For example, in \cite{GG'85} K. G{\l}azek and B. Gleichgewicht
calculated all ternary semigroups and groups for which their
ternary operation $f$ is a ternary polynomial over an infinite
commutative integral domain with identity. In the case when $f$ is
a ternary polynomial over an infinite commutative field $K$ the
operation $f$ has the form $f(x,y,z)=x+y+z+d$ or $f(x,y,z)=axyz$,
where $a\ne 0$ and $d$ are fixed. In the first case $(K,f)$ is a
ternary group, in the second $(K-\{0\},f)$.

\begin{theorem} {\bf (K.G{\l}azek, J.Michalski, 1984)}\newline
Let $m$ be odd and let $(G,f)$ be an $n$-ary group. Then the
operation $f$ has the form
\[
f(x_{1}^{n})=x_{1}\cdot x_{2}^{-1}\cdot x_{3}\cdot\ldots\cdot
x_{n-1}^{-1}\cdot x_{n},
\]
where $(G,\cdot)$ is an abelian group, if and only if $f$ is
idempotent and
\[
f\left(x_1^{i},y,y,x_{i+3}^{n}\right)=
f\left(x_1^{i},z,z,x_{i+3}^{n}\right)
\]
for all $\;0\leqslant i\leqslant n-2.$ In this case
$(G,\cdot)=ret_a(G,f)$ for some $a\in G$.
\end{theorem}

\begin{theorem} {\bf (K.G{\l}azek, J.Michalski, 1984)}\newline
Let $(G,\cdot)$ be a group and let \ $t_1,\ldots,t_n$ be fixed
integers. Then $G$ with the operation
\begin{equation*}
f(x_1^{m})=(x_1)^{t_{1}}\cdot(x_{2})^{t_{2}}\cdot\ldots \cdot
(x_{n-1})^{t_{n-1}}\cdot(x_{n})^{t_{n}},
\end{equation*}
is an $n$-ary group if and only if
\begin{enumerate}
\item[$(1)$] \ $x^{t_1}=x=x^{t_n}$,
\item[$(2)$] \ $t_j=k^{j}$ \ for some integer $k$ and all $j=2,\ldots,n-1$,
\item[$(3)$] \ $(x\cdot y)^k=x^k\cdot y^k.$
\end{enumerate}
\end{theorem}

In this case we say that $(G,f)$ is derived from the
$k$-exponential group. Basing on the results obtained in
\cite{DM'84} one can prove

\begin{proposition} {\bf (W.A.Dudek, K.G{\l}azek, 2006)}\newline
An $n$-ary group $(G,f)$ is derived from the $k$-exponential
$(k>0)$ group $(G,\cdot)$ if and only if in $(G,f)$ there exists
an idempotent $a$ such that
\[
f_{(k)}(\stackrel{(n-2)}{a},x,\stackrel{(n-2)}{a},x,\ldots,\stackrel{(n-2)}{a},x,a)=x
\]
for all $x\in G$. In this case $(G,\cdot)=ret_a(G,f)$.
\end{proposition}

The following theorem proved in \cite{DM'82} plays the fundamental
role in the calculation of non-isomorphic $n$-ary groups.

\begin{theorem}  {\bf (W.A.Dudek, J.Michalski, 1984)}\label{izoth} \newline
Two $n$-ary groups $(G_1,f_1)$, $(G_2,f_2)$ are isomorphic if and
only if for some $a\in G_1$ and $b\in G_2$ there exists an
isomorphism $h:ret_a(G_1,f_1)\to ret_b(G_2,f_2)$ such that
$$h(a)=b,$$
$$
h(f_1({\overline{a}},\ldots,{\overline{a}}))=
f_2(\overline{b},\ldots,\overline{b}),$$
$$
h(f_1(\overline{a},x,\!\overset{(n-2)}{a}))=
f_2(\overline{b},h(x),\!\overset{(n-2)}{b}). $$
\end{theorem}

\bigskip

This theorem reduces the problem of the calculation of
non-isomorphic $n$-ary groups to the classification of their
binary retracts. As an illustration of results obtained by K.
G{\l}azek and his co-workers we present the complete list of
$n$-ary groups $\langle\varphi,b\rangle$-derived from cyclic
groups.

Let $(\mathbb{Z}_{k},+)$ by the cyclic group modulo $k$. Consider
the following three $n$-ary operations:
\begin{equation*}
\arraycolsep=.5mm
\begin{array}{rl}
f_{a}(x_{1}^{n}) & \equiv (x_{1}+\ldots +x_{n}+a)\,(\mathrm{mod}\,k), \\[4pt]
g_{d}(x_{1}^{n}) & \equiv (x_{1}+dx_{2}+\ldots
+d^{n-2}x_{n-1}+x_{n})\,(\mathrm{mod}\,k), \\[4pt]
g_{d,c}(x_{1}^{n}) & \equiv (x_{1}+dx_{2}+\ldots
+d^{n-2}x_{n-1}+x_{n}+c)\,( \mathrm{mod}\,k),
\end{array}
\end{equation*}
where $a\in \mathbb{Z}_{k}$, \ $c,d\in \mathbb{Z}_{k}\setminus
\{0,1\},$ \ $d^{n-1}\equiv 1\,(\mathrm{mod}\,k).$ Additionally,
for the operation $g_{d,c}$ we assume that $dc\equiv
c\,(\mathrm{mod}\,k)$ holds. By Theorem~\ref{thGH},
$(\mathbb{Z}_{k},f_{a})$, $(\mathbb{Z}_{k},g_{d})$ and
$(\mathbb{Z}_{k},g_{d,c})$ are $n$-ary groups.

\begin{theorem} {\bf (K.G{\l}azek, J.Michalski, I.Sierocki,
1984)}\newline
 A $k$-element $n$-ary group $(G,f)$ is
$\langle\varphi,b\rangle$-derived from the cyclic group of order
$k$ if and only if it is isomorphic to exactly one $n$-ary group
of the form $(\mathbb{Z}_{k},f_{a})$, $(\mathbb{Z}_{k},g_{d})$ or
$(\mathbb{Z}_{k},g_{d,c})$, where $d|gcd(k,n-1)$ and $c|k.$
\end{theorem}

An infinite cyclic group can be identified with the group
$(\mathbb{Z},+)$. This group has only two automorphisms: $\varphi
(x)=x$ and $\varphi (x)=-x$. So, according to Theorem~\ref{thGH},
$n$-ary groups defined on $\mathbb{Z}$ have the form
$(\mathbb{Z},f_{a})$ or $(\mathbb{Z},g_{-1})$, where
\begin{equation*}
g_{-1}(x_{1}^{n})=x_{1}-x_{2}+x_{3}-x_{4}+\ldots +x_{n},
\end{equation*}
and $n$ is odd. Since $n$-ary groups $(\mathbb{Z},f_{a})$ and
$(\mathbb{Z},f_{b})$ are isomorphic for $a\equiv
b\,(\mathrm{mod}\,(n-1))$, we have $n-1$ non-isomorphic $n$-ary
groups of this form. Isomorphisms have the form $\varphi_{k}(x)=x-k$.\\

So, we have proved

\begin{theorem}  {\bf (W.A.Dudek, K.G{\l}azek, 2006)}\newline
An $n$-ary group $\langle\varphi,b\rangle$-derived from the
infinite cyclic group $(\mathbb{Z},+)$ is isomorphic to $n$-ary
group $(\mathbb{Z},f_a)$, where $0\leqslant a\leqslant (n-2)$, or
with $(\mathbb{Z},g_{-1})$, where $ n$ is odd.
\end{theorem}

A very similar result, but in incorrect form, was firstly
formulated in \cite{GMS}.

\medskip
Denote by $Inn\left(G,\cdot\right)$ the group of all inner
automorphisms of $(G,\cdot)$, by $Out\left(G,\cdot\right)$ the
factor group of $Aut\left(G,\cdot\right)$ by
$Inn\left(G,\cdot\right)$, and by $Out_{n}\left(G,\cdot\right)$
the set of all cosets $\overline{\gamma}\in
Out\left(G,\cdot\right)$ containing $\gamma$ such that
$\gamma^{n-1}\in Inn\left(G,\cdot\right)$. Then, the number of
pairwise non-isomorphic $n$-ary groups $\langle\varphi,
b\rangle$-derived from a centerless group $(G,\cdot)$, i.e., a
group for which $|Cent\left(G,\cdot\right)|=1$, is smallest or
equal to $s=\mid Out_{n}\left(G,\cdot\right)\mid $. It is equal to
$s$ if and only if $Out\left(G,\cdot \right)$ is abelian (for
details see \cite{GMS}).

For every $n$ and $k\ne 2,6$, there exists exactly one $n$-ary
group which is $\langle\varphi,b\rangle$-derived from $S_k$ (for
$k=2$ and $k=6$ we have one or two such $n$-ary groups relatively
to evenness of $n$).

Any finite group is uniquely determined by its multiplication
table which in fact is a Latin square. In the case of $n$-ary
groups the role of multiplication tables play $n$-dimensional
cubes. So, the problem of enumeration of all finite $n$-ary groups
can be reduced to the problem of enumeration of the corresponding
cubes. But it is rather difficult problem. The better approach was
suggested by K. G{\l}azek and J. Michalski. They proposed a method
based on the Hossz\'u-Gluskin Theorem and our Theorem~\ref{izoth},
i.e., the calculation of all non-isomorphic $n$-ary groups by the
classification of their retracts and automorphisms of these
retracts. Using this method they obtain in \cite{GM'84, GM'87,
GM'88} a full classification of all non-isomorphic $n$-ary groups
on sets with at most $7$ elements. The complete list of such
$n$-groups (with some comments) one can find in \cite{DG}.

\section{Covering group}

A binary group $(G^{\ast},\circ)$ is said to be a {\it covering
group} for the $n$-ary group $(G,f)$ if there exists an embedding
$\tau :G\rightarrow G^{\ast}$ such that $\tau(G)$ is a generating
set of $G^{\ast }$ and $\tau(f(x_1^n))= \tau
(x_1)\circ\tau(x_2)\circ\ldots\circ\tau(x_n)$ for every
$x_1,\ldots,x_n\in G.$ $(G^{\ast},\circ)$ is a \textit{universal
covering group} (or a \textit{free covering group}) if for any
covering group $(G^{\bullet},\diamond)$ there exists a
homomorphism from $G^*$ onto $G^{\bullet}$ such that the following
diagram is commutative:

\begin{center}
\begin{minipage}{6cm}

\ \ \ \ \ \ \ \ \ \ \ \ \ \ \ \ \ \ \ \ \ $G$

\ \ \ \ \ \ \ \ \ \ \ \ \ \ {  \ }$\swarrow ${  \ }$
\circlearrowleft ${  \ }$\searrow $

\ \ \ \ \ \ \ \ \ \ \ \ \ $G^{\ast}--\rightarrow \ G^\bullet$

\ \ \ \ \ \ \ \ \ \ \ \ \ \ \ \ \ \ \ \ {\footnotesize onto}
\end{minipage}
\end{center}

E. L. Post proved in \cite{post} that for every $n$-ary group
$(G,f)$ there exist a covering group $(G^{\ast},\circ)$ and its
normal subgroup $G_0$ such that $G^{\ast}/G_0$ is a cyclic group
of order $n-1$ and $f(x_1^n)=x_1\circ x_2\circ\ldots\circ x_n$ for
all $x_1,\ldots,x_n\in G$, where $G$ is identified with the
generator of the group $G^{\ast}/G_0$. So, the theory of $n$-ary
groups is closely related to the theory of {\it cyclic extensions
of groups}, but these theories are not equivalent.

Indeed, the above Post's theorem shows that for any $n$-ary group
$(G,f)$ we have the sequence
$$
O\longrightarrow (G_0,\circ)\longrightarrow (G^{\ast},\circ )
\stackrel{\zeta}\longrightarrow(\mathbb{Z}_n,+_n)\longrightarrow
O,
 $$
where $(G^{\ast },\circ)$ is the free covering group of $(G,f)$
and $G=\zeta^{-1}(1)$. We have also
\[\arraycolsep=.5mm
\begin{array}{ccccc}
&  & ( G_{1}^{\ast },\circ) &  &  \\
& \nearrow & \uparrow & \searrow &  \\
(G_0,\circ ) &  & \circlearrowright \;\;\; {\vert}
\circlearrowleft\circlearrowright &  & (\mathbb{Z}_n,+_n) \\
& \searrow &\downarrow & \nearrow &  \\
&  & (G_2^{\ast},\circ)&  &
\end{array}
\]
where we use

$\circlearrowright $ \ for the equivalence of extensions,

$\circlearrowleft $ \ for the isomorphism of suitable $n$-ary
groups.

\medskip

Of course, two $n$-ary groups determined in the above-mentioned
sense by two equivalent cyclic extensions are isomorphic. However,
as observed K. G{\l}azek, two non-equivalent cyclic extensions can
determine two isomorphic $n$-ary groups. Below we give an example
of two non-equivalent cyclic extensions induced by two isomorphic
$4$-ary groups.

\begin{example} {\sl (W.A.Dudek, K.G{\l}azek, 2006)}\newline
Consider two cyclic extensions of the cyclic group $\mathbb{Z}_3$
by $\mathbb{Z}_3$:
$$
0\longrightarrow\mathbb{Z}_3\stackrel{\alpha }{\longrightarrow
}\mathbb{Z}_9\stackrel{\beta_{1}}{\longrightarrow
}\mathbb{Z}_3\longrightarrow 0
$$
and
$$
0\longrightarrow \mathbb{Z}_3\stackrel{\alpha}{\longrightarrow
}\mathbb{Z}_9\stackrel{\beta_{2}}{\longrightarrow
}\mathbb{Z}_3\longrightarrow 0,
$$
where the homomorphisms $\alpha$, $\beta_{1}$ and $\beta_{2}$ are
given by:
\[
\begin{array}{lcl}
\;\alpha (x)=3x&{\rm for}& x\in\mathbb{Z}_3,\\[4pt]
\beta_{1}(x)= x({\rm mod}\,3)&{\rm for}&x\in\mathbb{Z}_9,\\[4pt]
\beta_{2}(x)= 2x({\rm mod}\,3)&{\rm for}&x\in\mathbb{Z}_9.
\end{array}
\]

It is easy to verify that the sets $\beta_{1}^{-1}(1)=\{1,4,7\}$
and $\beta_{1}^{-1}(1)=\{2,5,8\}$ with the operation
$g(x,y,z,v)=(x+y+z+v)({\rm mod}\,9)$ are $4$-ary groups. These
$4$-ary groups are isomorphic. The isomorphism
$\varphi:(\beta_{1}^{-1}(1),g)\longrightarrow(\beta_{2}^{-1}(1),g)$
has the form $\varphi(x)=2x({\rm mod}\,9)$. Nevertheless, the
above-mentioned extensions are not equivalent because there is no
automorphism $\lambda$ of $\mathbb{Z}_9$ such that $\lambda\circ
\alpha =\alpha$ and $\beta_{2}\circ\lambda =\beta_{1}$.
\end{example}

The Post's construction of a covering group for given $n$-ary
group $(G,f)$ is based on the following equivalence relation
defined on the set of all finite sequences of elements of $G$:
\[
x_1^k\sim y_1^m\Longleftrightarrow
f_{(\cdot)}(z_1^s,x_1^k,u_1^t)=f_{(\cdot)}(z_1^s,y_1^m,u_1^t)
\]
for some $z_1^s,u_1^t\in G$. Since $k\equiv m({\rm mod} (n-1))$,
each sequence is equivalent with some sequence of the length
$i=1,2,\ldots,n-1$. Defining on the set $G^{\ast}$ of such
obtained equivalence classes the operation $[x_1^i]\ast
[y_1^j]=[x_1^iy_1^j]$ we obtain a covering group of $(G,f)$ (for
details see \cite{post}).

Basing on this method K. G{\l}azek and B. Gleichgewicht proposed
in \cite{GG'73} (see also \cite{KG}) other more simple
construction of a covering group for ternary group. Namely, for a
ternary group $(G,f)$ they consider the set $G^2\cup G$ with the
operation
\[\arraycolsep=.5mm
\begin{array}{rl}
x\circ y&=\langle x,y\rangle\\
x\,\circ \langle y,z\rangle &=f(x,y,z)\\
\langle x,y\rangle\circ\, z&=f(x,y,z)\\
\langle x,y\rangle\circ\langle z,u\rangle &=\langle
f(x,y,z),u\rangle
\end{array}
\]
and the relation $\sim$ defined as follows:
\[\arraycolsep=.5mm
\begin{array}{rl}
x\sim y &\Longleftrightarrow x=y\\
\langle x,y\rangle\sim\langle x',y'\rangle &\Longleftrightarrow
(\exists a,b\in G)\;(f(x',a,b)=x\,\wedge\,f(a,b,y)=y).
\end{array}
 \]
Then $\sim$ is a congruence on the algebra $(G^2\cup G,\circ)$ and
$(G^2\cup G/\sim,\tilde{\circ})$, where $\tilde{\circ}$ is the
operation induced by $\circ$, is a universal covering group for
$(G,f)$. Obviously this construction can be extended to an
arbitrary $n$-ary group, but, as it was observed by W. A. Dudek at
the G{\l}azek's seminar in 1977, for $n>3$ this construction
coincides with the Post's construction. Nevertheless this
construction gives the impulse to the nice construction presented
by J. Michalski (cf. \cite{JM'81}).

\begin{theorem} {\bf (J.Michalski, 1981)}\newline
Let $c$ ba an arbitrary fixed element of an $n$-ary group $(G,f)$.
Then the set $G\times\mathbb{Z}_{n-1}$ with the operation
\[
\langle x,s\rangle\ast\langle y,t\rangle= \langle
f(x,\stackrel{(s)}{c},y,\stackrel{(t)}{c},\overline{c},\stackrel{(n-1-s\diamond
t)}{c}),s\diamond t\rangle ,
 \]
where $s\diamond t=(s+t+1)({\rm mod} (n-1))$, is a universal
covering group for $(G,f)$.
\end{theorem}

It is clear that $G\times\{n-2\}$ is a normal subgroup of
$G\times\mathbb{Z}_{n-1}$. The set of all classes (in the Post's
construction) induced by the sequences of the length $i+1$ can be
identified with $G\times\{i\}$. So, the set $G$ can be identified
with $G\times\{0\}$, i.e., $\tau(G)$ is a coset of
$G\times\mathbb{Z}_{n-1}$ modulo $G\times\{n-1\}$. Since for
$\tau(x)=\langle x,0\rangle$ we have
$\tau(f(x_1^n))=\tau(x_1)\ast\tau(x_2)\ast\ldots\ast\tau(x_n)$, an
$n$-ary group $(G,f)$ can be identified with an $n$-ary group
derived from $G\times\{0\}$.

In \cite{DM'82} (see also \cite{DM'84}) it is proved that for
every $c\in G$ the groups $ret_c(G,f)$ and $G\times\{n-1\}$ are
isomorphic.

\section{Representations}

Let $(G,g,\widetilde{\,}\;)$ be a general algebra with one $n$-ary
operation $g$ and one unary operation \ $\widetilde{}:G\to G$. If
$\psi,\varphi_1,\ldots,\varphi_n:G\to H$ are homomorphism of
$(G,g,\widetilde{}\;)$ into a semiabelian $n$-ary group
$(H;f,\bar{}\;)$, then, as it was observed by K. G{\l}azek and B.
Gleichgewicht in \cite{GG'77}, the mapping
\[
\widehat{f}(\varphi_1,\ldots,\varphi_n):G\to H
\]
defined by
 \[
\widehat{f}(\varphi_1,\ldots,\varphi_n)(x)=f(\varphi_1(x),\ldots,\varphi_n(x))
\]
and the mapping $\widehat{\psi}:G\to H$ defined by
 \[
 \widehat{\psi}(x)=\overline{\psi(x)}
 \]
are also homomorphisms of $(G,g,\widetilde{\,}\;)$ into
$(H;f,\bar{\,}\;)$. Moreover, the algebra
$(F;\widehat{f},\widehat{\,}\;)$ of all homomorphisms of
$(G,g,\widetilde{\,}\;)$ into $(H;f,\bar{\,}\;)$ belongs to the
same variety as the algebra $(H;f,\bar{\,}\;)$, so it is
semiabelian too. The set of all endomorphisms of a semiabelian
$n$-ary group forms an $(n,2)$-nearring $(E;\widehat{f},\circ )$
with unity, where $\circ$ is the superposition of endomorphisms.

Recall that an abstract algebra $(E;f,\cdot)$ is an
$(n,2)$-nearring if $(E;f)$ is an $n$-ary group, $(E;\cdot)$ is a
binary semigroup and the following two identities
\[
y\cdot f(x_1^n)=f(y\cdot x_1,y\cdot x_2,\ldots ,y\cdot x_n),
\]
\[
f(x_1^n)\cdot y=f(x_1\cdot y,x_2\cdot y,\ldots ,x_n\cdot y)
\]
hold. In the case when $(E;f)$ is a commutative $n$-ary group an
$(n,2)$-nearring is called an {\it $(n,2)$-ring}. Every
$(n,2)$-ring $(E;f,\cdot)$ with a cancellable element (with
respect to the multiplication $\cdot$ ) is isomorphic to an
$(n,2)$-ring of some endomorphisms of the $n$-ary group $(E;f)$
(cf. \cite{GG'77}).

Further study of homomorphisms of $n$-ary groups was continued by
A. M. Gal'mak (cf. for example \cite{gal'86, gal'01, gal'01a}).
Most of his results are based on the G{\l}azek's observation (cf.
\cite{KG'83} or \cite{KG'94}) that every weak homomorphism between
commutative $n$-ary groups is an ordinary homomorphism and the
following Post's construction of polyadic substitutions.

As it is well known an ordinary substitution, finite or infinite,
is a one-to-one map from the set $A$ onto $A$. Let
$A_1,A_2,\dots,A_{n-1}$ be a finite sequence of sets of the same
cardinality. The sequence
$\sigma=(\sigma_1,\sigma_2,\dots,\sigma_{n-1})$ of maps
$$
\sigma_1:A_1\to A_2, \ \ \sigma_2:A_2\to A_3,\;\dots,\;\sigma_{n-1}:A_{n-1}\to A_1
$$
is called a {\it $n$-ary substitution}. The superposition of two
$n$-ary substitutions is the sequence
$\tau=(\tau_1,\tau_2,\dots,\tau_{n-1})$ of maps
$$
\tau_1:A_1\to A_3, \ \ \tau_2:A_2\to A_4, \;\dots,\;\tau_{n-2}:A_{n-2}\to A_1, \ \
\tau_{n-1}:A_{n-1}\to A_2 .
$$
The set of $n$-ary substitutions is closed with respect to the
superposition of $n$ such substitutions. In fact it is an $n$-ary
group with respect to this operation. Moreover, each $n$-ary group
is isomorphic to the $n$-ary group of substitutions of some set
$A$ (cf. \cite{post}). Each $n$-ary group is isomorphic to the
$n$-ary group of some translations too \cite{gal'86}.

Let now $(A,f)$ be an $n$-ary group. The Cartesian product
$A^{n-1}$ endowed with the $n$-ary operation $g$ defined as the
skew product in the matrix $[a_{ij}]_{n\times (n-1)}$, i.e.,
 \[\begin{array}{lll}
g((a_{11},a_{12},\dots,a_{1\,n-1}),(a_{21},a_{22},\dots,a_{2\,n-1}),\dots,
(a_{n1},a_{n2},\dots,a_{n\,n-1}))\\[4pt]
=(f(a_{11},a_{22},a_{33},\dots,a_{n-1\,n-1},a_{n1}),
f(a_{12},a_{23},a_{34},\dots,a_{n-2\,n-1},a_{n-1\,1},a_{n2}),\\[4pt]
\hfill\ldots,f(a_{1\,n-1},a_{21},a_{32},\dots,a_{n-1\,n-2}a_{n\,n-1})),\end{array}
\]
is an $n$-ary group (cf. \cite{post}) which is called {\it
diagonal}. The diagonal $n$-ary group of invertible linear
transformations of a complex vector space is used in \cite{GWW} to
the description of one-dimensional representations of cyclic
$n$-groups. The invariant subspaces of a representation $\rho$,
sequences of $\rho$-invariant subspaces, the covering
representation ${\widehat\rho}$ and the relations between $\rho$
and ${\widehat\rho}$ are discussed in \cite{WW'84}.

Matrix representations of ternary groups are described in
\cite{BDD2} (see also \cite{BDD1}). In these representations each
matrix is determined by two elements. Below we present the general
concept of such representations.

Let $V$ be a complex vector space and $\operatorname*{End}V$
denotes a set of $\mathbb{C}-$linear endomorphisms of $V$.

\begin{definition}
\label{def-left} A \textit{left bi-element representation} of an
$n$-ary group $(G,f)$ in a vector space $V$ is a map
$\Pi^{L}:G^{n-1}\rightarrow\operatorname*{End}V$ such that
\begin{align}
\Pi^{L}\left(a_{1}^{n-1}\right)\circ\Pi^{L}\left(b_{1}^{n-1}\right)
& =\Pi^{L}\left( f(a_{1}^{n-1},b_{1}),b_{2}^{n-1}\right)
,\label{lr1}\\
\Pi^{L}\big(\stackrel{(n-2)}{a},\overline{a}\,\big)   &
=\operatorname*{id}\nolimits_{V} \label{lr2}
\end{align}
for all $a,a_1,\ldots,a_{n-1},b_1,\ldots,b_{n-1}\in G$.
\end{definition}

Note that the axioms considered in the above definition are the
natural ones satisfied by left multiplications $x\mapsto
f(a_1^{n-1},x)$.

Using (\ref{lr2}) and the associativity of the operation $f$ it is
not difficult to see that
\[
\Pi^{L}\left( f(a_1^{n-1}) ,a_{n}^{2n-2}\right) =\Pi^{L}\left(
a_{1}^{i},f(a_{i+1}^{n+i-1}),a_{n+i}^{2n-2}\right)
\]
for all $a_1,\ldots,a_{2n-2}\in G$ and $i=1,2,\ldots,n-1$.
Moreover, from (\ref{lr2}), by (\ref{dor-l}), we have also
\[
\Pi^{L}\big(\stackrel{(j-1)}{a},\overline{a},\stackrel{(n-j-1)}{a}\big)
=\operatorname*{id}\nolimits_{V}
\]
for all $a\in G$ and $j=1,2,\ldots,n-1$.

\begin{lemma}\label{lem20}
A left bi-element representation of an $n$-ary group is uniquely
determined by two elements.
\end{lemma}
\begin{proof} According to the definition of an $n$-ary group for
every $a,a_1,\ldots,a_{n-1},b\in G$ there exists $c\in G$ such
that $f(a_1^{n-1},a)=f(\stackrel{(n-2)}{b},c,a)$. So,
\[\arraycolsep=.5mm
\begin{array}{rl}
\Pi^{L}\left(a_{1}^{n-1}\right)&=
\Pi^{L}\left(a_{1}^{n-1}\right)\circ\Pi^{L}\big(\stackrel{(n-2)}{a},\overline{a}\big)=
\Pi^{L}\big(f(a_{1}^{n-1},a),\stackrel{(n-3)}{a},\overline{a}\big)\\[4pt]
&=\Pi^{L}\big(f(\stackrel{(n-2)}{b},c,a),\stackrel{(n-3)}{a},\overline{a}\big)=
\Pi^{L}\big(\stackrel{(n-2)}{b},c\big)\circ\Pi^{L}\big(\stackrel{(n-2)}{a},\overline{a}\big)\\[4pt]
&=\Pi^{L}\big(\stackrel{(n-2)}{b},c\big),
\end{array}
\]
which completes the proof.
\end{proof}

\begin{corollary}\label{cor21}
$\Pi^{L}(a_1^{n-1})=\Pi^{L}(b_1^{n-1})\Longleftrightarrow
f(a_1^{n-1},a)=f(b_1^{n-1},a)$ \ $\forall a\in G$.
\end{corollary}
\begin{proposition}
Let $(G,f)$ be an $n$-ary group derived from a binary group
$(G,\odot)$. There is one-to-one correspondence between
representations of $(G,\odot)$ and left bi-element representations
of $(G,f)$.
\end{proposition}

\begin{proof}
Because $(G,f)$ is derived from $(G,\odot)$, then $x\odot
y=f\big(x,\stackrel{(n-2)}{e},y\big)$ and $\overline{e}=e$, where
$e$ is the identity of $(G,\odot)$. If $\pi$ is a representation
of $(G,\odot)$, then (as it is not difficult to see)
\begin{equation}\label{lrep}
\Pi^{L}(x_1^{n-1})
=\pi(x_1)\circ\pi(x_2)\circ\ldots\circ\pi(x_{n-1})
 \end{equation}
is a left bi-element representation of $(G,f)$. Conversely, if
$\Pi^{L}$ is a left bi-element representation of $(G,f)$, then
$\pi(x)=\Pi^{L}\big(\stackrel{(n-2)}{e},x\big)$ is a
representation of $(G,\odot)$ and (\ref{lrep}) is satisfied.
\end{proof}

\begin{proposition}
Any left bi-element representations of an $n$-ary group $(G,f)$
induces a representation of its retract.
\end{proposition}
\begin{proof}
Let $(G,\odot)=ret_a(G,f)$ for some fixed $a\in G$. According to
Lemma~\ref{lem20}, for all $a_1,\ldots,a_{n-1}\in G$ there exists
$c\in G$ such that
$\Pi^{L}(a_1^{n-1})=\Pi^{L}\big(\stackrel{(n-2)}{a},c\big)=\pi(c)$.
Then $\pi(\overline{a})=\operatorname*{id}\nolimits_{V}$ and
\[\arraycolsep=.5mm
\begin{array}{rl}
\pi(x)\circ\pi(y)&=\Pi^{L}\big(\stackrel{(n-2)}{a},x\big)\circ
\Pi^{L}\big(\stackrel{(n-2)}{a},y\big)=\Pi^{L}\big(f\big(\stackrel{(n-2)}{a},
x,a\big),\stackrel{(n-3)}{a},y\big)\\[4pt]
&=\Pi^{L}\big(\stackrel{(n-2)}{a},
f\big(x,\stackrel{(n-2)}{a},y\big)\big)=
\Pi^{L}\big(\stackrel{(n-2)}{a},x\odot y\big)=\pi(x\odot y),
\end{array}
\]
which proves that $\pi$ is a representation of $(G,\odot)$.
\end{proof}

For ternary groups also the converse statement is true: every
representation of $ret_a(G,f)$ induces a left bi-element
representation of $(G,f)$ (cf. \cite{BDD2}). All such bi-element
representations are invertible.

\begin{definition}
A \textit{right bi-element representation} of an $n$-ary group
$(G,f)$ in $V$ is a map
$\Pi^{R}:G^{n-1}\rightarrow\operatorname*{End}\,V$ such that
\begin{align*}
\Pi^{R}\left(a_{1}^{n-1}\right)\circ\Pi^{R}\left(
b_{1}^{n-1}\right) &
=\Pi^{R}\left(a_{1}^{n-2},f(a_{n-1},b_{1}^{n-1}) \right) ,\\
\Pi^{R}\big(\stackrel{(n-2)}{a},\overline{a}\,\big)   &
=\operatorname*{id}\nolimits_{V}
\end{align*}
for all $a,a_1,\ldots,a_{n-1},b_1,\ldots,b_{n-1}\in G$.
\end{definition}

It is clear that left and right bi-element representations are
dual.

\begin{example}\label{exam-kg}
Let $(G,f)$ be an $n$-ary group and $\mathbb{C}G$ denote a vector
space spanned by $G$. It means that any element $u$ of
$\mathbb{C}G$ can be uniquely presented in the form
$u=\sum_{i=1}^{m}k_{i}y_{i}$, with $k_{i}\in\mathbb{C},$ $y_{i}\in
G$, $m\in N$ (we do not assume that $G$ has finite rank).
Moreover, $\mathbb{C}G$ is an $n$-ary (group) algebra (for details
see \cite{zek} or \cite{zek/art3}). Then left and right bi-element
regular representations can be immediately defined by means of
this structure
\begin{align}
\Pi_{reg}^{L}\left(a_{1}^{n-1}\right)u& =\sum_{i=1}^{m}k_{i}f(a_{1}^{n-1},y_{i}),\label{pr1}\\
\Pi_{reg}^{R}\left(a_{1}^{n-1}\right)u&
=\sum_{i=1}^{m}k_{i}f(y_{i},a_{1}^{n-1}). \label{pr2}
\end{align}
\end{example}

Left (right) regular bi-element representations of a ternary group
are unitary \cite{BDD2}.
\begin{example}
Consider an $n$-ary group $(G,f)$, where $n$ is odd,
$G=\mathbb{Z}_{3}=\{ 0,1,2\}$ and
$f(x_1^n)=\sum_{i=1}^n(-1)^{i+1}x_i(\rm{mod}\,3)$. It is clear
that
\[
\Pi^{L}(x_1^{n-1})=\Pi^{R}(x_{n-1},x_{n-2},\ldots,x_2,x_1).
\]

From the proof of Lemma~\ref{lem20} it follows that it is
sufficient to consider the representations of the form
$\Pi^{L}\big(\stackrel{(n-2)}{a},b\big)$. By Corollary~\ref{cor21}
in our case
\[
\Pi^{L}(\stackrel{(n-2)}{a},b)=\Pi^{L}(\stackrel{(n-2)}{c},d)\Longleftrightarrow
(a-b)=(c-d)(\rm{mod}\,3).
\]
Straightforward calculations give the left regular representation
in the manifest matrix form, where $\Pi^{L}(a,b)$ means
$\Pi^{L}\big(\stackrel{(n-2)}{a},b\big)$.

\begin{align*}
\Pi_{reg}^{L}\left(  0,0\right)   &  =\Pi_{reg}^{L}\left(
2,2\right) =\Pi_{reg}^{L}\left(  1,1\right)  = \left(
\begin{array}
[c]{ccc}%
1 & 0 & 0\\
0 & 1 & 0\\
0 & 0 & 1
\end{array}
\right)  =[1]\oplus\lbrack1]\oplus\lbrack1],\\[4pt]
\Pi_{reg}^{L}\left(2,0\right)&=\Pi_{reg}^{L}\left(1,2\right)
=\Pi_{reg}^{L}\left(0,1\right)=\left(
\begin{array}
[c]{ccc}%
0 & 1 & 0\\
0 & 0 & 1\\
1 & 0 & 0
\end{array}
\right) \\
&  =[1]\oplus\left(
\begin{array}
[c]{cc}%
-\dfrac{1}{2} & -\dfrac{\sqrt{3}}{2}\\
\dfrac{\sqrt{3}}{2} & -\dfrac{1}{2}%
\end{array}
\right)  =[1]\oplus\left[
-\dfrac{1}{2}+\dfrac{1}{2}i\sqrt{3}\right]
\oplus\left[  -\dfrac{1}{2}-\dfrac{1}{2}i\sqrt{3}\right]  ,\\[4pt]
\Pi_{reg}^{L}\left(  2,1\right)   &  =\Pi_{reg}^{L}\left(
1,0\right) =\Pi_{reg}^{L}\left(  0,2\right)  =\left(
\begin{array}
[c]{ccc}%
0 & 0 & 1\\
1 & 0 & 0\\
0 & 1 & 0
\end{array}
\right) \\
&  =[1]\oplus\left(
\begin{array}
[c]{cc}%
-\dfrac{1}{2} & \dfrac{\sqrt{3}}{2}\\
-\dfrac{\sqrt{3}}{2} & -\dfrac{1}{2}%
\end{array}
\right)  =[1]\oplus\left[
-\dfrac{1}{2}-\dfrac{1}{2}i\sqrt{3}\right] \oplus\left[
-\dfrac{1}{2}+\dfrac{1}{2}i\sqrt{3}\right]  .
\end{align*}
\end{example}

Observe that for an $n$-ary (group) algebra $\mathbb{C}G$ from
Example~\ref{exam-kg} we can consider additionally middle
representations of the form
\[
\Pi_{reg}^{M}(a_1^{n-1})u=\sum_{i=1}^{m}k_if(a_1^{j-1},y_i,a_{j}^{n-1}),
\]
where $j=2,3,\ldots,n-1$ is fixed.

Since in this case we obtain complicated formulas we restrict our
attention to middle bi-element representations of ternary groups
described in \cite{BDD1} and \cite{BDD2}.

\begin{definition}
A \textit{middle bi-element representation} of a ternary group
$(G,f)$ in $V$ is a map $\Pi^{M}:G\times
G\rightarrow\operatorname*{End}\,V$ such that
\begin{equation*}
\Pi^{M}\left(a_{3},b_{3}\right)\circ\Pi^{M}\left(
a_{2},b_{2}\right)\circ\Pi^{M}\left(a_{1},b_{1}\right)
=\Pi^{M}\left(f(a_{3},a_{2},a_{1}),f(b_{1},b_{2},b_{3})\right)  ,
\label{pm}
\end{equation*}
\begin{equation*}
\Pi^{M}\left(a,b\right)\circ\Pi^{M}\left(\overline{a},\overline
{b}\right)=\Pi^{M}\left(\overline{a},\overline{b}\right)\circ
\Pi^{M}\left(a,b\right)=\operatorname*{id}\nolimits_{V}
\label{pm1}
\end{equation*}
for all $a,a_{1},a_{2},a_{3},b,b_{1},b_{2},b_{3}\in G$.
\end{definition}

The composition of two middle bi-element representations is not a
middle representation, but in some cases described in \cite{BDD2}
it is a left bi-element representation. Obviously,
$\Pi^{M}(a,b)=\Pi^{M}(c,d)$ if and only if $f(a,y,b)=f(c,y,d)$ for
every $y\in G$.

\begin{example}\label{exam-z3m}
Let $G=\mathbb{Z}_3$ and $f(x,y,z)=(x-y+z)(\rm{mod}\,3)$. Then
$(G,f)$ is a ternary group in which
\[
\Pi^{M}(a,b)=\Pi^{M}(c,d)\Longleftrightarrow
(a+b)=(c+d)(\rm{mod}\,3).
\]
So,

\begin{align*}
\Pi_{reg}^{M}\left(  0,0\right)   &  =\Pi_{reg}^{M}\left(
1,2\right) =\Pi_{reg}^{M}\left(  2,1\right)  =\left(
\begin{array}
[c]{ccc}%
1 & 0 & 0\\
0 & 0 & 1\\
0 & 1 & 0
\end{array}
\right)=\left[  1\right]  \oplus\left[
\begin{array}
[c]{cc}%
-1 & 0\\
0 & 1
\end{array}
\right]   ,\\
\Pi_{reg}^{M}\left(  0,1\right)   &  =\Pi_{reg}^{M}\left(
1,0\right) =\Pi_{reg}^{M}\left(  2,2\right)  =\left(
\begin{array}
[c]{ccc}%
0 & 1 & 0\\
1 & 0 & 0\\
0 & 0 & 1
\end{array}
\right) =\left[  1\right]  \oplus\left[
\begin{array}
[c]{cc}%
\dfrac{1}{2} & -\dfrac{\sqrt{3}}{2}\\
-\dfrac{\sqrt{3}}{2} & -\dfrac{1}{2}%
\end{array}
\right]  ,\\
\Pi_{reg}^{M}\left(  0,2\right)   &  =\Pi_{reg}^{M}\left(
2,0\right) =\Pi_{reg}^{M}\left(  1,1\right)  =\left(
\begin{array}
[c]{ccc}%
0 & 0 & 1\\
0 & 1 & 0\\
1 & 0 & 0
\end{array}
\right) =\left[  1\right]  \oplus\left[
\begin{array}
[c]{cc}%
\dfrac{1}{2} & \dfrac{\sqrt{3}}{2}\\
\dfrac{\sqrt{3}}{2} & -\dfrac{1}{2}%
\end{array}
\right]  .
\end{align*}
This representation $\Pi_{reg}^{M}$ is equivalent to the
orthogonal direct sum of two irreducible representations, i.e.,
one-dimensional trivial  and two-dimensional.
\end{example}

In this example for all $x$ we have
$\,\Pi^{M}(x,\overline{x})=\Pi^{M}(x,x)\neq\mathrm{id}_{V}$, but
$\,\Pi^{M}(x,y)\circ\Pi^{M}(x,y)=\mathrm{id}_{V}$ for all $x,y$.

\medskip

Putting in the above example $\gamma_{k+l}^{M}=\Pi_{reg}^{M}(k,l)$
for $\;k,l\in\mathbb{Z}_{3}$, we obtain the identity
\[
\gamma_{i}^{M}\circ\gamma_{j}^{M}\circ\gamma_{k}^{M}=\gamma_{f(i,j,k)}
^{M},\;\;\;i,j,k\in\mathbb{Z}_{3},
\]
which in some sense can be treated as a \textit{ternary analog of
Clifford algebra} \cite{BDD1}. Any matrix representation of this
identity gives rise to the middle representation $\,\Pi^{M}\left(
k,l\right)=\gamma_{k+l}$.

\medskip

Different connections between middle bi-element representations of
a ternary group $(G,f)$ and representations of its retract and
covering group are described in \cite{BDD2}. For example, any
middle bi-element representation of a ternary group $(G,f)$
derived from a group $(G,\cdot)$ has the form
$\Pi^{M}(a,b)=\pi(a)\circ\rho(b^{-1})$, where
$\pi(x)=\Pi^{M}(x,e)$ and $\rho(x)=\Pi^{M}(e,\overline{x})$ are
pairwise commuting representations of $(G,\cdot)$.

\section{$\mathcal{Q}$-independent sets in $HG$-algebras}

Let $\mathfrak{A}=(A,\mathbb{F})$ be an algebra $\emptyset \neq
X\subseteq A$. The set $X$ is said to be
$\mathcal{M}$\textit{-independent} if
\[
\begin{array}{c}
(\forall n\in \mathbb{N}, \ n\leqslant card(X)) \ (\forall
f,g\in\mathbb{T}^{(n)}(\mathfrak{A})) \ (\forall\underset{\neq
}{\underbrace{a_{1},\ldots ,a_{n}}}\in X)\\
\big[f(a_{1}^{n})=g(a_{1}^{n})\Longrightarrow f=g \big].
\end{array}
\]

\medskip

This condition is equivalent to each of the following ones:
\medskip
\begin{enumerate}
\item[(a)]  $(\forall n\in \mathbb{N}$, $n\leqslant card(X))$
$(\forall f,g\in\mathbb{T}^{(n)}(\mathfrak{A}))$ $(\forall p:
X\rightarrow A)$ $(\forall a_{1},\ldots ,a_{n}\in X)$
$$
\big[f(a_{1}^{n})=g(a_{1}^{n})\Longrightarrow f(p(a_{1}),\ldots
,p(a_{n}))=g(p(a_{1}),\ldots ,p(a_{n}))\big],
$$

\item[(b)]  $(\forall p\in A^{X}) \ (\exists\bar{p}\in Hom(\langle
X\rangle_{\mathfrak{A}},\mathfrak{A})) \ \bar{p}|_{X}=p$, where
$\langle X\rangle_{\mathfrak{A}}$ is a subalgebra of
$\mathfrak{A}$ generated by $X$.
\end{enumerate}

\medskip
The notion of $\mathcal{M}$-independence is stronger than that of
independence with respect to the closure operator of such a kind
$X\mapsto\langle X\rangle_{\mathfrak{A}}$ (for $X\subseteq A$).

Let $\emptyset\neq X\subseteq A$ and
$$
\mathcal{Q}_X\subseteq A^X=\mathcal{M}_X=\{p\;|\;p:X\rightarrow
A\},
$$
$$
\mathcal{Q}(A)=\mathcal{Q}=\bigcup\{\mathcal{Q}_X \;|\;X\subseteq
A \},
$$
$$
\mathcal{M}(A)=\mathcal{M}=\bigcup\{A^X \;|\;X\subseteq A\}.
$$

For an algebra $\mathfrak{A} = (A,\mathbb{F})$, a mapping \ $p:X
\rightarrow A$ belongs to $\mathcal{H}_X (\mathfrak{A})$ if and
only if there exists a homomorphism \ $\bar{p}:\langle X\rangle
_{\mathfrak{A}}\rightarrow A$ such that $\bar{p}|_{X}=p$.

The set $X$ is said to be $\mathcal{Q}$-{\it independent} if
$\mathcal{Q}_X \subseteq \mathcal{H}_X (\mathfrak{A})$ or,
equivalently,
$$
(\forall p\in \mathcal{Q}_{X})\;(\forall\;{\rm finite }\;
n\leqslant card(X))\;(\forall f,g\in
\mathbb{T}^{(n)}(\mathfrak{A}))\; (\forall a_{1},\ldots ,a_{n}\in
X)\;$$
$$\big[f(a_{1}^{n})=g(a_{1}^{n})\Longrightarrow f(p(a_{1}),\ldots
,p(a_{n}))=g(p(a_{1}),\ldots ,p(a_{n}))\big].$$

In the case when $\mathcal{Q}=\bigcup \{p|_{X}\;|\;p\in A^{A}, \
X\subseteq A\}$ and
$$
(\forall f,g\in \mathbb{T}^{(1)}(\mathfrak{A}))\;(\forall a\in A)
\;\big[f(a)=g(a)\Longrightarrow f(p(a))=g(p(a))\big],$$ the set
$X$ is said to be {\it $\mathcal{G}$-independent}.
\medskip

For commutative groups, the notion of $\mathcal{G}$-independence
gives us the well-known \textit{linear independence}.

For $HG$-algebras of type $\mathfrak{H}=(G,+,\varphi,b)$, where
$(G,+)$ is a commutative group, the equality
\[
F_1(x_1,\ldots,x_m)=F_2(x_1,\ldots,x_m)
\]
(for two term operations of the form (\ref{sum}) in
$\mathfrak{H}$) is equivalent to the equality
\[
H(x_1,\ldots,x_m)=0,
\]
where $H\in\mathbb{T}^{(m)}(\mathfrak{H})$, i.e.,
$H(x_1,\ldots,x_m)=\sum\limits_{i=1}^{m}g_i(x_i)+k_{_H} b$, and
$0$ denotes the zero of the group $(G,+)$.

Consider a subset $X$ of $G$. Let for $a_1,\ldots,a_m\in X$ the
equality
\[
H(a_1,\ldots,a_m)=0
\]
hold. Taking into account the mapping $p:X\to\langle
X\rangle_{\mathfrak{A}}$ defined by $p(a_i)=0$ and $p(x)=x$ for
$x\in X\setminus\{a_1,\ldots,a_m\}$, we get $k_{_H}b=0$. Therefore
$$
\sum\limits_{i=1}^{m} g_i(a_i)=0.
$$
Consider the mapping $q_j:X\to \langle X\rangle_{\mathfrak{A}}$
defined for fixed $j\in\{1,\ldots,m\}$ as follows:
\[
q_j(x)=\left\{\begin{array}{ccl} a_j&{\rm if }&x=a_j ,\\[4pt]
0&{\rm if }&x\ne a_j.\end{array}\right.
\]
We obtain $g_j(a_j)=0$ for all $j=1,2,\ldots,m$. (In the
considered case all $q_j$ belong to $\mathcal{M}$ and
$\mathcal{G}$ .)

In particular, we can easily observe, by similar considerations,
that the following result holds:

\begin{theorem} {\bf (W.A.Dudek, K.G{\l}azek, 2006)}\newline
Let $X\subseteq G$ be a subset of the $HG$-algebra
$\mathfrak{H}=(G,+,\varphi,b)$. Then $X$ is
$\mathcal{G}$-independent if and only if for any $m\leqslant
card(X)$ for all $a_1,\ldots,a_m\in X$ and every term operation
$H(x_1,\ldots,x_m)=\sum\limits_{i=1}^{m}g_i(x_i)+k_{_H}b$ the
equality
\begin{equation}\label{ff}
\sum\limits_{i=1}^{m}g_i(a_i)+k_{_H}b=0
\end{equation}
is equivalent with
\[
\left(\forall i\in\{1,\ldots,m\}\right)\,\left(g_i(a)=0\;\&\;
k_{_H}b=0\right).
\]

Moreover, $X$ is $\mathcal{M}$-independent in this $HG$-algebra if
and only if for all pairwise different elements $a_1,\ldots,a_m$
from $X$ equality $(\ref{ff})$ implies $g_i(x)=0$ for all
$i=1,2,\ldots,m$ and $\,k_Hb=0$.
\end{theorem}

\end{document}